\documentclass{amsart}
\usepackage{amssymb}
\usepackage{verbatim}
\usepackage{array}
\usepackage{latexsym}
\usepackage{enumerate}
\usepackage{tikz}
\usepackage{float}
\usepackage{amsmath}
\usepackage{amsfonts}
\usepackage{amsthm}
\usepackage{mathrsfs}
\usepackage{color}
\usepackage{cancel}
\usepackage{longtable}
\usepackage[english]{babel}

\newtheorem{theorem}{Theorem}[section]
\newtheorem{proposition}[theorem]{Proposition}
\newtheorem{lemma}[theorem]{Lemma}
\newtheorem{corollary}[theorem]{Corollary}

\newtheorem{example}[theorem]{Example}

\newtheorem{remark}[theorem]{Remark}

\def\cC{\mathcal C}
\def\cD{\mathcal D}

\def\cL{\mathcal L}

\def\cP{\mathcal P}

\def\cS{\mathcal S}

\def\PG{{\rm{PG}}}

\newcommand{\PSL}{\mbox{\rm PSL}}
\newcommand{\PGL}{\mbox{\rm PGL}}

\newcommand{\MaxInt}{{\rm MaxInt}}

\title{On two M\"obius functions for a finite non-solvable group}
\date{}

\author{Francesca Dalla Volta}
\address{Dipartimento di Matematica e Applicazioni, Universit\`a degli Studi di Milano-Bicocca, via Roberto Cozzi 55, 20125 Milano, Italy}
\email {francesca.dallavolta@unimib.it}

\author{Giovanni Zini}
\address{Dipartimento di Matematica e Fisica, Universit\`a degli Studi della Campania ''Luigi Vanvitelli'', viale Lincoln 5, 81100 Caserta, Italy}
\email {giovanni.zini@unicampania.it}

\begin{document}



\begin{abstract}
Let $G$ be a finite group, $\mu$ be the M\"obius function on the subgroup lattice of $G$, and $\lambda$ be the M\"obius function on the poset of conjugacy classes of subgroups of $G$. It was proved by Pahlings that, whenever $G$ is solvable, the property $\mu(H,G)=[N_{G^\prime}(H):G^{\prime}\cap H]\cdot\lambda(H,G)$ holds for any subgroup $H$ of $G$. It is known that this property does not hold in general; for instance it does not hold for every simple groups, the Mathieu group $M_{12}$ being a counterexample. In this paper we investigate the relation between $\mu$ and $\lambda$ for some classes of  non-solvable groups; among them, the minimal non-solvable groups. We also provide several examples of groups not satisfying the property.
\end{abstract}

\maketitle

\begin{small}



\end{small}

\section{Introduction}

The M\"obius function of a locally finite poset $\cP$ is the function $\mu_{\cP}:\cP\times\cP\to\mathbb{Z}$ defined by
$$
\mu_{\cP}(x,y)=0\;\textrm{if}\;x\not\leq y,\quad \mu_{\cP}(x,x)=1,\quad \mu_{\cP}(x,y)=-\sum_{x<z\leq y}\mu_{\cP}(z,y)\;\textrm{if}\; x<y.
$$
Let $G$ be a finite group and $\cL$ be the subgroup lattice of $G$.  
The M\"obius function of $G$ is defined as $\mu:~\cL~\to~\mathbb{Z}$,
$ H\mapsto\mu_{\cL}(H,G)$, and it is simply denoted by $\mu(H)$. 
The M\"obius function $\mu$ of a finite group $G$ was considered by Hall \cite{Hall} in order to enumerate generating tuples of elements of $G$. Actually, the M\"obius function of $G$ has many other applications in different areas of mathematics, in the context of enumerative problems where the M\"obius inversion formula (\cite[Proposition 3.7.1]{Stanley}) turns out to be applicable. 
These areas include group theory, graph theory, algebraic topology, computer science; for a detailed description, see \cite{Castillo,DownsJones1987,Liu,Quillen} and the references therein.

For instance, beginning with Hall \cite{Hall}, the M\"obius function of  $\cL$ is involved in problems related to the probability of generating a finite group $G$ by a given number of elements. Later, Mann \cite{Mann1} introduced the following complex series for any profinite group $G$:
$$P_G(s)=\sum _{H\leq_{o}G}\frac {\mu (H)}{[G:H]^s},$$
where $H$ ranges over all open subgroups of $G$. Mann conjectured
that this sum is absolutely convergent in some half complex plane whenever $G$ is a positively finitely generated (PFG) group.
The conjecture is nowadays reduced to the following one (see \cite{Mann2,Lucchini}):
there exist $c_1,c_2\in\mathbb{N}$ such that, for any almost simple group $G$, $|\mu(H)|\leq [G:H]^{c_1}$ for any $H<G$; and, for any $n\in\mathbb{N}$, the number of subgroups $H<G$ of index $n$ in $G$ satisfying $G=H\,{\rm soc}(G)$ and $\mu(H)\ne0$ is upper bounded by $n^{c_2}$.

One could try to attach this problem looking for bounds for a different M\"obius function:
instead of the lattice ${\cL}$, we consider the poset $\cC$ of conjugacy classes of subgroups of $G$, where $[H]\leq[K]$ if and only if $H\leq K^g$ for some $g\leq G$; its M\"obius function $\cC\to\mathbb{Z}$, $[H]\mapsto\mu_{\cC}([H],[G])$ is denoted by $\lambda(H)$. Clearly, $\cC$ is the image of $\cL$ under the order-preserving map $H\mapsto[H]$, and $\cC$ is in general more tractable than $\cL$; see the comments in \cite[Section 3.1]{BDL}.

%

Hawkes, Isaacs and \"Ozaydin \cite{HIO} showed that
$$\mu(\{1\})=|G^\prime|\cdot\lambda(\{1\})$$
holds for any finite solvable group $G$, and later Pahlings \cite{Pahlings} proved that
\begin{equation}\label{eq:mulambda}
\mu(H)=[N_{G^\prime}(H):G^\prime\cap H]\cdot\lambda(H)
\end{equation}
holds for any $H\leq G$ whenever $G$ is finite and solvable.
We say that $G$ satisfies the $(\mu,\lambda)$-property if Equation \eqref{eq:mulambda} holds for any $H\leq G$.
It is known that the $(\mu,\lambda)$-property does not hold for every finite group; for instance, it does not hold for the Mathieu group $M_{12}$ (see \cite{Bianchi}) and for the unitary groups $U_3(2^{2^n})$ (see \cite{Zini}).
Neverthless, finding relations between the functions $\mu$ and $\lambda$ may be of some interest, also because of the above conjectures.


In this paper we extend the work of Pahlings about the $(\mu,\lambda)$-property for finite groups.
In Section \ref{sec:minnonsolv} we show that the $(\mu,\lambda)$-property is valid for some infinite families of non-solvable groups, beginning in Theorem \ref{prop:minsimple} with minimal simple groups. This is then generalized in Theorem \ref{th:minimalnonsolvable} to groups which are not solvable but close to be, that is, to all minimal non-solvable groups (i.e. non-solvable groups whose proper subgroups are all solvable).
We also present other infinite families of non-solvable groups satisfying the $(\mu,\lambda)$-property; namely, the simple groups $L_2(q)$, $Sz(q)$, $R(q)$, and the almost simple groups ${\rm PGL}_2(q)$.
In Section \ref{sec:product} we provide sufficient conditions for a direct product of finite groups to satisfy the $(\mu,\lambda)$-property;
we also point out relations about the $(\mu,\lambda)$-property between a finite group and its extensions.
Finally, Section \ref{sec:examples} contains examples of group satisfying and of groups not satisfying the $(\mu,\lambda)$-property.

\section{Notations and preliminaries}

Throughout the paper $G$ is a finite group, $\mu(H,G)$ is the M\"obius function of $H\leq G$ in the subgroup lattice $\cL$ of $G$, and $\lambda(H,G)$ is the M\"obius function of the conjugacy class $[H]$ of $H\leq G$ in the poset $\cC$ of conjugacy classes of subgroups of $G$, ordered as follows: $[H]\leq[K]$ if $H\leq K^g$ for some $g\in G$.
We denote these functions also by $\mu(H)$ and $\lambda(H)$, when the group $G$ is clear.
The notation $\mu(n)$ stays for the number-theoretic M\"obius function on the positive integers, which coincides with $\mu_{\cD}(1,n)$ in the divisibility poset $\cD$ on $\mathbb{N}$; we will use the number-theoretical M\"obius function in our computations in Tables \ref{tabella:l2qpari} to \ref{tabella:ree}.
Also, we point out that we mostly follow the group-theoretical notation of the ATLAS \cite{ATLAS}.

We denote by $\MaxInt(G)$ the set whose elements are $G$ and the subgroups of $G$ which are intersection of maximal subgroups of $G$.
The knowledge of $\MaxInt(G)$ is essential in the study of $\mu$, as Lemma \ref{maxintmu} shows.

\begin{lemma}[{\rm Hall \cite[Theorem 2.3]{Hall}}]\label{maxintmu}
Let $H\leq G$ be such that $\mu(H)\ne0$. Then $H\in\MaxInt(G)$.
\end{lemma}
Lemma \ref{maxintmu} follows from the fact that $\mathcal{L}$ is a poset, and in particular that two elements $H_1, H_2\in\mathcal{L}$ have an meet $H_1\cap H_2$.
Note that in general the poset $\cC$ is not a lattice and the meet of $[H_1]$ and $[H_2]$ does not always exist.
Neverthless, we will prove the analogue of Lemma \ref{maxintmu} also for $\cC$.

We say that {\bf $G$ satisfies the $(\mu,\lambda)$-property} if the following relation holds for any $H\leq G$:
$$ \mu(H)=[N_{G^\prime}(H):G^\prime \cap H]\cdot\lambda(H). $$


Recall that a {\it minimal non-solvable group} is a a non-solvable group all of whose proper subgroups are solvable.
A minimal non-solvable group which is simple is called {\it minimal simple}.
We point out that, if $G$ is a minimal non-solvable group and $\Phi(G)$ is its Frattini subgroup, then $G/\Phi(G)$ is minimal simple; hence, minimal simple groups are exactly the Frattini-free minimal non-solvable groups.

Finite minimal simple groups are classified as follows.


\begin{lemma}[{\rm Thompson \cite{Thompson}}]\label{lemma:minsimple}
The finite minimal simple groups are the following ones:
\begin{itemize}
\item $L_2(2^r)$, where $r$ is a prime;
\item $L_2(3^r)$, where $r$ is an odd prime;
\item $L_2(p)$, where $p>3$ is a prime such that $5\mid(p^2+1)$;
\item $Sz(2^r)$, where $r$ is an odd prime;
\item $L_3(3)$.
\end{itemize}
\end{lemma}


\section{$(\mu,\lambda)$-property for some families of non-solvable groups}\label{sec:minnonsolv}

\subsection{$(\mu,\lambda)$-property for minimal non-solvable groups}\label{subsec:minnonsolv}

Here we show that the $(\mu,\lambda)$-property holds for many classes of finite groups, among which all minimal non-solvable groups.
A key role in our proofs (as well as in the explicit computations of the $\lambda$ function for any group) is played by Lemma \ref{maxintlambda}.


\begin{lemma}\label{maxintlambda}
Let $H\leq G$ be such that $\lambda(H)\ne0$. Then $H\in\MaxInt(G)$.
\end{lemma}

\begin{proof}
We use induction on $[G:H]$. The result is clearly true for $H=G$.
Let $H<G$ and let $K\in\MaxInt(G)$ be the intersection of all maximal subgroups of $G$ which contain $H$; thus, $K\leq M$ for any $M\in\MaxInt(G)$ containing $H$.
Suppose $H\notin\MaxInt(G)$, so that $H<K$.
Let $N\leq G$ be such that $H<N$ and $\lambda(N)\ne0$. By induction $N\in\MaxInt$, and hence $K\leq N$.
Therefore,
$$ \lambda(H)=-\sum_{[H]<[N]\leq[G],\,\lambda(N)\ne0}\lambda(N)=-\sum_{[K]\leq[N]\leq[G]}\lambda(N)=0. $$
\end{proof}

From Lemmas \ref{maxintmu} and \ref{maxintlambda} follows that only subgroups of $G$ containing $\Phi(G)$ have to be considered:

\begin{corollary}\label{cor:frattini}
If $H\leq G$ and $\Phi(G)\not\leq H$, then $\mu(H)=\lambda(H)=0$.
\end{corollary}

Theorem \ref{prop:minsimple} proves the $(\mu,\lambda)$-property for minimal simple groups. Clearly, if $G$ is simple then $G=G^\prime$ and the $(\mu,\lambda)$-property for $H\leq G$ reads $\mu(H)=[N_G(H):H]\lambda(H)$.

\begin{theorem}\label{prop:minsimple}
Let $G$ be a minimal simple group. Then $G$ satisfies the $(\mu,\lambda)$-property.
\end{theorem}

\begin{proof}
By Lemma \ref{lemma:minsimple}, either $G= L_2(q)$ for some $q$, and the claim follows from Proposition \ref{lemma:l2q} below; or $G= Sz(q)$ for some $q$, and the claim follows from Proposition \ref{lemma:szq} below; or $G=L_3(3)$, and the claim follows by direct inspection with Magma \cite{MAGMA}.
\end{proof}

The proofs of the next propositions are quite technical. We go into details in the proof of Proposition \ref{lemma:szq}, while we just sketch the proofs of the remaining Propositions \ref{lemma:l2q}, \ref{prop:PGL(2,q)} and \ref{lemma:ree}.

Note that Propositions \ref{lemma:szq} and \ref{lemma:l2q} refer to larger classes of simple groups than just the minimal ones.

\begin{proposition}\label{lemma:szq}
Let $q=2^e$ for some odd $e\geq3$, and $G$ be the simple Suzuki group $Sz(q)$. Then $G$ satisfies the $(\mu,\lambda)$-property.
\end{proposition}

\begin{proof}
Downs and Jones computed $\MaxInt(G)$ in \cite[Theorem 6]{DownsJones2016}, as well as the normalizer $N_G(H)$ (\cite[Table 2]{DownsJones2016}) and the value $\mu(H)$ (\cite[Table 1]{DownsJones2016}) for any $H\in\MaxInt(G)$.
Hence, we compute $\lambda(H)$.
By Lemma \ref{maxintlambda}, we restric to the subgroups $H\in\MaxInt(G)$.
The results are summarized in Table \ref{tabella:suzuki}; if a subgroup $H<G$ does not appear in Table \ref{tabella:suzuki}, then $\mu(H)=\lambda(H)=0$.
We use the results and the techniques performed in \cite{DownsJones2016}; these techniques rely on a detailed analysis of the subgroups of $G$ and the $2$-transitive action of $G$ on the Suzuki-Tits ovoid $\Omega$ of $\PG(3,q)$, as described by Suzuki \cite{Suzuki1962}.

The same notation as in \cite{DownsJones2016} is used. 
In particular, for any divisor $h\mid e$, $G(h)$ denotes a subgroup $Sz(2^h)$; the stabilizer in $G(h)$ of one point $\infty\in\Omega$ is denoted by  $F(h)$ and acts as a Frobenius group on $\Omega$; $F(h)=Q(h)A_0(h)\cong(E_{2^h})^{1+1}:C_{2^h-1}$, where $Q(h)\cong(E_{2^h})^{1+1}$ and $A_0(h)\cong C_{2^h-1}$ are respectively the Frobenius kernel and a Frobenius complement in $F(h)$; $Z(h)\cong E_{2^h}$ is the center of $Q(h)$ if $h>1$, while $C_2\cong Z(1)<Q(1)\cong C_4$.
Also, $B_0(h)$ is a dihedral group of order $2(2^h-1)$ normalizing $A_0(h)$; for $h>1$, $N_G(A_0(h))=B_0(e)$.
We denote by $A_1(h),A_2(h)\leq G(h)_{\infty}$ two cyclic groups such that $\{|A_1(h)|,|A_2(h)|\}=\{2^h+2^{(h+1)/2}+1,2^h-2^{(h+1)/2}+1\}$ with $5\mid|A_1|$, $5\nmid|A_2|$; for $i=1,2$, $B_i(h)$ is a subgroup of $G(h)$ with $B_i(h)\cong A_i(h):C_4$. We have $B_1(h)=N_G(A_1(h))$ and, if $h>1$, then $B_2(h)=N_G(A_2(h))$; also, $F(1)=B_2(1)\cong C_4$, $B_0(1)\cong C_2$, $G(1)=B_1(1)\cong{\rm AGL_1(5)}$, and $A_0(1)\cong \{1\}$.

By \cite[Theorem 6]{DownsJones2016}, every $H\in\MaxInt(G)$ is conjugated to one of the groups
$$G(h),\quad F(h),\quad Q(h),\quad Z(h),\quad B_i(h),\quad A_i(h),\quad i\in\{0,1,2\},\quad h\mid e;$$
each of them yields a single conjugacy class in $G$.
\begin{enumerate}
\item First, consider the case $h>1$.
\begin{itemize}
\item Suppose  that $H$ is conjugated to one of the following groups:
$$G(h),\quad B_1(h),\quad F(h),\quad B_0(h),\quad B_2(h).$$
Then $H=N_G(H)$ by \cite[Theorem 9]{DownsJones2016} and hence $[N_{G}(H):H]=1$.
By \cite[Table 2]{DownsJones2016}, if $K\in\MaxInt(G)$ satisfies $H\leq K$, then $K$ is the unique element of its conjugacy class $[K]$ which contains $H$. This implies that the subposets $\{K\in\cL\colon K\in\MaxInt(G),\,H\leq K\leq G\}$ and $\{[K]\in\cC\colon K\in\MaxInt(G),\,[H]\leq[K]\leq[G]\}$ are isomorphic. This implies $\mu(H)=\lambda(H)$. Thus, the $(\mu,\lambda)$-property holds for $H$.
\item Suppose that $H$ is conjugated to $Q(h)$ for some $h$.
By \cite[Tables 2]{DownsJones2016}, if  $h\mid k $ the overgroups of $H$ in $\MaxInt(G)$ are conjugated to $G(k)$ or to $F(k)$ or to $Q(k)$ (obviously, $Q(k)\ne Q(h)$ implies $k\ne h$). Since $\lambda(G(k))=\mu(e/k)$ and $\lambda(F(k))=-\mu(e/k)$, we have $\lambda(H)=0$ by induction. By \cite[Table 1]{DownsJones2016}, the $(\mu,\lambda)$-property holds for $H$.
\item Suppose that $H$ is conjugated to $Z(h)$.
By \cite[Table 2]{DownsJones2016}, the overgroups of $H$ in $\MaxInt(G)$ are conjugated to $G(k)$ or to $F(k)$ or to $Q(k)$, with $h\mid k$; or to $Z(k)$, with $h\mid k$ and $h\ne k$.
By \cite[Table 1]{DownsJones2016} and induction, we have that $\lambda(H)=0$ and the $(\mu,\lambda)$-property holds for $H$.
\item Suppose that $H$ is conjugated to $A_0(h)$.
By \cite[Theorem 9]{DownsJones2016}, we have $[N_{G}(H):H]=[B_0(e):A_0(h)]=\frac{2(2^e-1)}{2^h-1}$.
By \cite[Table 2]{DownsJones2016}, the overgroups of $H$ in $\MaxInt(G)$ are conjugated to $G(k)$ or to $F(k)$ or to $B_0(k)$, with $h\mid k$; or to $A_0(k)$, with $h\mid k$ and $h\ne k$.
Since $\lambda(G(k))+\lambda(F(k))=0$ and $\lambda(B_0(k))=-\mu(e/k)$, we have by induction $\lambda(H)=\mu(e(h))$.
By \cite[Table 1]{DownsJones2016} we have $\mu(A_0(h))=\frac{2(2^e-1)}{2^h-1}\mu(e/h)$; thus, the $(\mu,\lambda)$-property holds for $H$.
\item Suppose that $H$ is conjugated to $A_1(h)$ for some $h$.
By \cite[Tables 1 and 2]{DownsJones2016}, the overgroups of $H$ in $\MaxInt(G)$ are conjugated either to $G(k)$ with $h\mid k$, satisfying $\lambda(G(k))=\mu(e/k)$; or to $B_1(k)$ with $h\mid k$, satisfying $\lambda(B_1(k))=-\mu(e/k)$; or to $A_1(k)$ with $h\mid k$ and $h\ne k$. By induction, we have $\lambda(H)=0$.
As $\mu(H)=0$, the $(\mu,\lambda)$-property holds for $H$.
\item Suppose that $H$ is conjugated to $A_2(h)$.
By \cite[Table 2]{DownsJones2016}, the overgroups of $H$ in $\MaxInt(G)$ are conjugated either to $G(k)$ with $h\mid k$, satisfying $\lambda(G(k))=\mu(e/k)$; or to $B_2(k)$ with $h\mid k$, satisfying $\lambda(B_2(k))=-\mu(e/k)$; or to $A_2(k)$ with $h\mid k$ and $h\ne k$. By induction, we have $\lambda(H)=0$.
As $\mu(H)=0$ by \cite[Table 1]{DownsJones2016}, the $(\mu,\lambda)$-property holds for $H$.
\end {itemize}
\item Now, consider the case $h=1$.
\begin{itemize}
\item Suppose that $H$ is conjugated to $A_1(1)$ or to $G(1)=B_1(1)$. Then the proof is analogue to the case $h>1$.
\item Suppose that $H$ is conjugated to $F(1)=Q(1)=B_2(1)\cong C_4$.
By \cite[Theorem 9]{DownsJones2016}, we have $[N_{G}(H):H]=\frac{q}{2}$.
By \cite[Table 2]{DownsJones2016}, the overgroups of $H$ in $\MaxInt(G)$ are conjugated either to $G(k)$ with $k\mid e$, satisfying $\lambda(G(k))=\mu(e/k)$; or to $F(k)$ with $1\ne k\mid e$, satisfying $\lambda(F(k))=-\mu(e/k)$; or to $Q(k)$ with $1\ne k\mid e$, satisfying $\lambda(Q(k))=0$; or to $B_1(k)$ with $k\mid e$, satisfying $\lambda(B_1(k))=-\mu(e/k)$; or to $B_2(k)$ with $1\ne k\mid e$, satisfying $\lambda(B_2(k))=-\mu(e/k)$.
The subposets $\{[K]\colon K\in\MaxInt(G)\,, G(1)\leq K\leq G(e)\}$ and $\{[K]\colon K\in\MaxInt(G)\,, G(1)\leq K\leq G(e)\}$ are both isomorphic to the divisibility poset $\mathcal{D}_e=\{n\in\mathbb{N}\colon 1\mid n\mid e\}$ having a minimum and a maximum element; thus, $\sum_{1\mid k\mid e}\lambda(G(k))=\sum_{1\mid k\mid e}\lambda(B_1(k))=\sum_{1\mid k\mid e}\mu(e/k)=0$.
The subposet $\{[F(k)]\colon 1\ne k\mid e\}$ is isomorphic to the divisibility poset $\mathcal{D}_e$ minus its minimum element $1$; therefore, $\sum_{1\ne k\mid e}\lambda(F(k))=\sum_{1\ne k\mid e}-\mu(e/k)=+\mu(e/1)=\mu(e)$. In the same way, we have $\sum_{1\ne k\mid e}\lambda(B_2(k))=\mu(e)$.
Hence, $\lambda(H)=-(\mu(e)+\mu(e))=-2\mu(e)$.
By \cite[Table 1]{DownsJones2016}, $\mu(H)=-2^e \mu(e)$. Therefore, the $(\mu,\lambda)$-property holds for $H$.
\item Suppose that $H$ is conjugated to $Z(1)=B_0(1)\cong C_2$.
By \cite[Theorem 9]{DownsJones2016}, we have $[N_{G}(H):H]=\frac{q^2}{2}$.
When $K$ runs over the overgroups of $H$ having a subgroup isomorphic to $C_4$, the conjugacy classes $[K]$ form a subposet of $\cC$ with minimum element $[B_2(1)]$; hence, the sum $\sum \lambda(K)$ over this set is equal to zero.
The overgroups of $H$ in $\MaxInt(G)$ without subgroups isomorphic to $C_4$ are the groups conjugated to $B_0(k)$ for some $k\mid e$ with $k\ne 1$.
Since $\lambda(B_0(k))=-\mu(e/k)$, and the subposets $\{[B_0(k)]\colon 1\mid k\mid e\}$ and $\mathcal{D}_e$ are isomorphic, we have that $\lambda(H)=-\sum_{1\ne k\mid e}\lambda(B_0(k))=-\mu(e/1)=-\mu(e)$.
By \cite[Table 1]{DownsJones2016}, $\mu(H)=-\frac{q^2}{2}\mu(e)$. Therefore, the $(\mu,\lambda)$-property holds for $H$.
\item Suppose that $H=A_2(1)=A_0(1)=\{1\}$.
As in the previous point, we may compute
$$\lambda(H)=-\sum_{[K]\in\cC\,,K\in\MaxInt(G)\,,2\nmid|K|}\lambda(K),$$
because the sublattice of $\cC$ given by the subgroups of even order of $G$ has a minimum element $[B_0(1)]$, $B_0(1)\cong C_2$.
Hence, we only consider the subgroups $A_i(k)$.
For $i\in\{1,2\}$, we have $\lambda(A_i(k))=0$ for any overgroup $A_i(k)$ of $H$.
The subgroups $A_0(k)$ form a sublattice of $\cC$ isomorphic to $\mathcal{D}_e$, and satisfy $\lambda(A_0(k))=\mu(e/k)$ for any $k\ne1$.
Therefore, $\lambda(A_0(1))=-\sum_{1\ne k \mid e}\mu(e/k))=\mu(e/1)=\mu(e)$.
Since $\mu(H)=|G|\mu(e)$ by \cite[Table 1]{DownsJones2016}, the $(\mu,\lambda)$-property holds for $H$.
\end{itemize}
\end{enumerate}
\end{proof}

\begin{proposition}\label{lemma:l2q}
Let $q\geq4$ be a prime power and $G=L_2(q)$. Then $G$ satisfies the $(\mu,\lambda)$-property.
\end{proposition}

\begin{proof}
For any $H\leq G$, $\mu(H)$ was computed by Downs in \cite{Downs}. When $q$ is prime, $\lambda(H)$ has been computed in \cite{Pahlings}, and the claim was already proved in \cite[Proposition 3]{Pahlings} (pay attention to a misprint in the proof of \cite[Proposition 3]{Pahlings}: the right value of $a_n$ is $1$ or $2$ according respectively to $p\equiv\pm1$ or $p\not\equiv\pm1$ modulo $4n$, not modulo $2n$ as it is written).

Hence, we restrict to the computation of $\lambda(H)$ with $q$ non-prime.
The value $\mu(C_3)$ is not computed correctly in \cite{Downs} when $q=3^e$; the right value is $\mu(C_3)=\frac{q}{3}$ for $e>2$, and is given in Table \ref{tabella:l2qdispariecc}.
Also, the case $q=25$ is not considered in \cite{Downs}; yet, in the cases $q=9$ and $q=25$ the claim follows by direct computation with Magma \cite{MAGMA}.


The techniques used are similar to the ones used in \cite{DownsThesis,Downs,Pahlings}, together with Lemma \ref{maxintlambda}. The explicit computations are quite long and follow closely the steps performed in the cited papers; therefore, we have chosen to omit them and to summarize the results in Tables \ref{tabella:l2qpari}, \ref{tabella:l2qdispari}, \ref{tabella:l2qdispariecc}, \ref{tabella:l2qdispariquad}. If a subgroup $H<G$ does not appear in the tables and is not a maximal subgroup of $G$, then $\mu(H)=\lambda(H)=0$.

The same notation as in \cite{Downs} is used. In particular, $q=p^e$ with $p$ prime and $e>1$. For any divisor $1\ne h\mid e$, $r(h)=\frac{p^h-1}{2}$ and $s(h)=\frac{p^h+1}{2}$; $\mathcal{G}_h=\PGL_2(p^h)$; $\mathcal{S}_h=\PSL_2(p^h)$; $A_n$ and $S_n$ are alternating and symmetric groups of degree $n$; $C_m,D_m,E_m$ are respectively cyclic, dihedral, elementary abelian groups of order $m$; for any $1\ne d\mid (p^h-1)$, $M_{h,d}\cong E_{p^h}:C_{d}$ is contained in a subgroup $\PGL_2(p^h)$ or $L_2(p^h)$ of $L_2(q)$ and stabilizes a point in the natural action on the projective line over $\mathbb{F}_q$.
\end{proof}

We conclude this section proving the validity of the $(\mu,\lambda)$-property for minimal non-solvable groups. 

\begin{theorem}\label{th:minimalnonsolvable}
Let $G$ be a minimal non-solvable group. Then $G$ satisfies the $(\mu,\lambda)$-property.
\end{theorem}

\begin{proof}
If $\Phi(G)=1$, then $G$ is minimal simple and Theorem \ref{prop:minsimple} proves the claim.
Let $\Phi(G)\ne1$ and $H\leq G$. If $\Phi(G)\not\leq H$, the claim follows by Corollary \ref{cor:frattini}.
If $\Phi(G)\leq H$, then let $\bar{G}=G/\Phi(G)$ and $\bar{H}=H/\Phi(G)$. As in the proof of the theorem in \cite{Pahlings}, we have $\mu(H,G)=\mu(\bar{H},\bar{G})$, $\lambda(H,G)=\lambda(\bar{H},\bar{G})$, and $[N_{{\bar{G}}^\prime}(\bar{H}):\bar{H}\cap {\bar{G}}^\prime]=[N_{G^\prime}(H):H\cap G^\prime]$. As $\bar{G}$ is minimal simple, the result follows from Proposition \ref{prop:minsimple}.
\end{proof}

\subsection{$(\mu,\lambda)$-property for other families of groups}\label{sec:other}

In this section, we collect the results about the validity of $(\mu,\lambda)$-property for other groups, namely the almost simple groups ${\rm PGL}_2(q)$ and the simple Ree groups $R(q)$.

\begin{proposition}\label{prop:PGL(2,q)}
For any prime power $q$, the group ${\rm PGL}_2(q)$ satisfies the $(\mu,\lambda)$-property.
\end{proposition}

We omit the proof of Proposition \ref{prop:PGL(2,q)}; here, we also omit the tables, which are very similar to the ones of $L_2(q)$.
Note that ${\rm PGL}_2(q)=L_2(q)$ when $q$ is even.
Also, the arguments are simplified when $q$ is odd by the fact that for any $H\leq {\rm PGL}_2(q)$ there is just one conjugacy class in ${\rm PGL}_2(q)$ of subgroups isomorphic to $H$, except when $H$ is cyclic of order $2$ or a dihedral group $D_{2m}$ with $2m$ dividing either $ \frac{q-1}{2}$ or  $\frac{q+1}{2}$.

\begin{lemma}\label{lemma:ree}
Let $q=3^e$ with $e\geq3$ odd. The simple small Ree group $R(q)$ satisfies the $(\mu,\lambda)$-property.
\end{lemma}

\begin{proof}
Pierro exhibited in \cite{Pierro} a set $M$ of subgroups of $G=R(q)$ with $\MaxInt(G)\subseteq M$ (\cite[Table 4]{Pierro}), as well as the normalizer $N_G(H)$ (\cite[Lemmas 2.5 to 2.8]{Pierro}) and the value $\mu(H)$ (\cite[Theorem 1.11]{Pierro}) for any $H\in M$.
Hence, we compute $\lambda(H)$. The results are summarized in Table \ref{tabella:ree}. 
We omit the explicit computations, which are similar to the ones used in \cite{Pierro}, except that for the use of Lemma \ref{maxintlambda}.
They rely on a detailed analysis of the subgroups of $G$ and the $2$-transitive action of $G$ on the Ree-Tits ovoid $\Omega$ of $\PG(6,q)$, as described by Tits \cite{Tits}.

The notation is the same as in \cite{Pierro}, which is the same as in the ATLAS \cite{ATLAS}. Note that, for any $H$ in Table \ref{tabella:ree}, the isomorphism type of $H$ identifies a unique conjugacy class in $G$.
\end{proof}

\section{Products and Extensions}\label{sec:product}

In this Section, we consider the $(\mu,\lambda)$-property for direct products of a a finite number of finite groups and for finite extensions of a finite group.

Proposition \ref{lemma:muprodotto} gives a sufficient condition for the $\mu$ and $\lambda$ functions of a direct product to split on the correspondent functions of the factors; it generalizes \cite[Result 2.8]{Hall}.


\begin{proposition}\label{lemma:muprodotto}
Let $n\geq2$ and $G=\prod_{i=1}^n G_i$ be a direct product of groups $G_i$'s such that every maximal subgroup $M$ of $G$ splits as a direct product $M=\prod_{i=1}^n M_i$, with $M_i\leq G_i$ for any $i$.
Let $H=\prod_{i=1}^n H_i \leq G$ with $H_i\leq G_i$ for any $i$.
Then
$$
\mu_G(H) = \prod_{i=1}^n \mu_{G_i}(H_i),\qquad
\lambda_G(H) = \prod_{i=1}^n \lambda_{G_i}(H_i).
$$
\end{proposition}

\begin{proof}
From the assumptions follows immediately that, if $K\in\MaxInt(G)$, then $K=\prod_{i=1}^n K_i$ with $K_i\leq G_i$ for any $i$.
Hence, in the computation of $\mu(H)$ and $\lambda(H)$ we only consider the groups $H<\prod_{i=1}^n K_i \leq G$ with $K_i\leq G_i$ for any $i$.
Let $I\subseteq\{1,\ldots,n\}$ be such that $H_i\ne G_i$ for $i\in I$, and $H_i=G_i$ for $i\in\{1,\ldots,n\}\setminus I$. Then the subposet of $\cL$ made by the groups $K=\prod_{i=1}^n K_i$ satisfying $H\leq K\leq G$ is isomorphic to the subgroup poset of groups $K=\prod_{i\in I}K_i$ satisfying $\prod_{i\in I}H_i\leq K\leq \prod_{i\in I}G_i$; analogous poset isomorphism holds for the posets of conjugacy classes.
Hence, $\mu_G(H)=\mu_{\prod_{i\in I}G_i}(\prod_{i\in I}H_i)$ and $\lambda_G(H)=\lambda_{\prod_{i\in I}G_i}(\prod_{i\in I}H_i)$.
Then we can assume that $H_i\ne G_i$ for all $i=1,\ldots,n$.
We use induction on $[G:H]$, the claim being true when $H$ is a maximal subgroup of $G$. We have
\begin{equation}\label{eq:muprodotto}
\mu_G \left(\prod_{i=1}^n H_i \right) = -\sum_{H<\prod_{i=1}^n K_i\leq G}\mu_G \left(\prod_{i=1}^n K_i \right) = -\sum_{H<\prod_{i=1}^n K_i\leq G} \prod_{i=1}^n \mu_{G_i}(K_i).
\end{equation}
Fix $i_0\in\{1,\ldots,n\}$, and for any $i\ne i_0$ fix $K_i$ with $H_i<K_i\leq G_i$. Then
$$
\sum_{H_{i_0}<K_{i_0}\leq G_{i_0}} \prod_{i=1}^n \mu_{G_i}(K_i) = 
\left(\prod_{i\ne i_0}\mu_{G_i}(K_i) \right)\cdot\left( \sum_{H_{i_0}<K_{i_0}\leq G_{i_0}}\mu_{G_{i_0}}(K_{i_0}) \right) = -\mu_{G_{i_0}}(H_{i_0})\cdot\prod_{i\ne i_0} \mu_{G_i}(K_i).
$$
We apply this argument to Equation \eqref{eq:muprodotto}, dividing the summation according to the distinct $\epsilon$ indexes $i_h\in\{1,\ldots,n\}$ such that $K_{i_h}$ varies with $H_{i_h}<K_{i_h}\leq G_{i_h}$, while $K_{j_h}=H_{j_h}$ for the remaining $n-\epsilon$ indexes $j_h$.
We obtain
$$
\sum_{ H<\prod_{i=1}^n K_i \leq G,\; K_{j_1}=H_{j_1},\ldots,K_{j_{n-\epsilon}}=H_{j_{n-\epsilon}}} \prod_{i=1}^n \mu_{G_{i}}(K_i) = (-1)^\epsilon \prod_{i=1}^n \mu_{G_i}(H_i).
$$
Therefore,
$$
\mu_G\left(\prod_{i=1}^n H_i\right) = -\left(\prod_{i=1}^n \mu_{G_i}(H_i)\right)\cdot\left(+\sum_{k\in\{1,\ldots,n\},\; k\;\textrm{even}}\binom{n}{k} -\sum_{k\in\{1,\ldots,n\},\; k\;\textrm{odd}}\binom{n}{k}\right)=\prod_{i=1}^n \mu_{G_i}(H_i),
$$
where the last equality was obtained using $\binom{n}{k}=\binom{n-1}{k}+\binom{n-1}{k-1}$. The claim on $\lambda_G(H)$ follows similarly.
\end{proof}

We now apply Proposition \ref{lemma:muprodotto} to the direct product of groups satisfying the $(\mu,\lambda)$-property.

\begin{proposition}\label{prop:prodotto}
Let $G=\prod_{i=1}^n G_i$ be a direct product of groups $G_i$'s such that every maximal subgroup $M$ of $G$ splits as a direct product $M=\prod_{i=1}^n M_i$, with $M_i\leq G_i$ for any $i$.
If $G_1,\ldots,G_n$ satisfy the $(\mu,\lambda)$-property, then $G$ satisfies the $(\mu,\lambda)$-property.
\end{proposition}

\begin{proof}
If $H\in\MaxInt(G)$, then $H=\prod_{i=1}^n H_i$ with $H_i\leq G_i$ for any $i$.
In fact, for any $j=1,\ldots,m$, let $M_j$ be a maximal subgroup of $G$ such that $M_{j}=\prod_{i=1}^n M_{j,i}$ with $M_{j,i}\leq G_{i}$ for any $i=1,\ldots,n$; then $\cap_{j=1}^m M_j =\prod_{i=1}^n \cap_{j=1}^m M_{i,j}$.
We have $G^\prime =\prod_{i=1}^n G_i^\prime$ and $G^\prime \cap H = \prod_{i=1}^n G_i^\prime \cap H_i$.
This implies $[N_{G^\prime}(H):G^\prime \cap H] = \prod_{i=1}^n [N_{G_i^\prime}(H_i): G_i^\prime \cap H_i]$.
Now the claim follows using Lemma \ref{lemma:muprodotto}.
\end{proof}

\begin{example}
Let $G=G_1\times G_2$ where no non-trivial quotients of $G_1$ and $G_2$ are isomorphic.
Then it is easily seen from Goursat's lemma that every maximal subgroup $M$ of $G$ splits as $M=M_1\times M_2$ with $M_i\leq G_i$, $i=1,2$.

For instance, Proposition \ref{prop:prodotto} applies to $G=G_1\times G_2$ where $G_1$ is minimal non-solvable and $G_2$ is solvable, or $G_1$ and $G_2$ are minimal non-solvable with non-isomorphic Frattini quotients.
\end{example}

About extensions of groups which satisfy the $(\mu,\lambda)$-property, we give the following elementary remark, which follows from the same arguments as in the proof of Theorem \ref{th:minimalnonsolvable}. 

\begin{remark}
The $(\mu,\lambda)$-property holds for every Frattini extension of a group satisfying the $(\mu,\lambda)$-property.
\end{remark}

Remark \ref{remark:extensions} shows how to get groups not satisfying the $(\mu,\lambda)$-property from a given one.

\begin{remark}\label{remark:extensions}
If a group $G$ is a finite extension of a group $\bar G$ which does not satisfy the $(\mu,\lambda)$-property, then $G$ does not satisfy the $(\mu,\lambda)$-property.
\end{remark}

In fact, arguing as in the proof of Theorem \ref{th:minimalnonsolvable}, it is clear that a subgroup $\bar H$ for which the $(\mu,\lambda)$-property fails in $\bar G$ is the homomorphic image of a subgroup $H$ for which the $(\mu,\lambda)$-property fails in $G$.

\section{Final remarks}
\label{sec:examples}

We start this section by considering a class of finite groups which contains the minimal non-solvable ones, namely the finite non-solvable N-groups; recall that an N-group is a group all of whose local subgroups are solvable.
Finite non-solvable N-groups were classified by Thompson \cite{Thompson}; they are almost simple groups $G$, where $S\leq G\leq{\rm Aut}(S)$ and $S$ is one of the following simple groups:
\begin{itemize}
\item the linear group $L_2(q)$, for some prime power $q\geq4$;
\item the Suzuki group $Sz(q)$, for some non-square power $q\geq8$ of $2$;
\item the linear group $L_3(3)$;
\item the unitary group $U_3(3)$;
\item the alternating group $A_7$;
\item the Mathieu group $M_{11}$;
\item the Tits group $^2 F_4(2)^{\prime}$.
\end{itemize}




\begin{proposition}\label{remark:Ngroups}
Every finite simple N-group $G$ other than $U_3(3)$ satisfies the $(\mu,\lambda)$-property, unlike $U_3(3)$.
\end{proposition}

\begin{proof}
Unless $G$ is cyclic of prime order, $G$ is contained in the above list.
If $G=L_2(q)$ or $G=Sz(q)$, the claim follows by Propositions \ref{lemma:szq} and \ref{lemma:l2q}. By direct inspection with Magma \cite{MAGMA} follows that the groups $L_3(3)$, $A_7$, $M_{11}$ and $^2 F_4(2)^{\prime}$ satisfy the $(\mu,\lambda)$-property, while $U_3(3)$ does not.
\end{proof}

By performing some computations for groups of small order, we notice the following facts:
\begin{itemize}
\item all groups of order at most $2000$ satisfy the $(\mu,\lambda)$-property;
\item the group $U_3(3)$, of order $6048$, is the smallest simple group which does not satisfy the $(\mu,\lambda)$-property;
\item the subgroups $H\leq U_3(3)$ at which the $(\mu,\lambda)$-property fails are those isomorphic to $C_2$, $S_3$, $D_8$, or $S_4$.
\end{itemize}

We do not know whether $U_3(3)$ is the smallest non-solvable group not satisfying the $(\mu,\lambda)$-property.

The inspection of $U_3(3)$ may give hints about the structure of groups for which the $(\mu,\lambda)$-property fails, as explained by the following remark.

\begin{remark}\label{remark:U33}
Let $G$ be a finite group, $H$ be a subgroup of $G$, $\cS=\{K\leq G : H\leq K\}$ be the subposet of the subgroup lattice of $G$ made by the overgroups of $H$, and $\bar{\cS}=\{[K]\leq[G] : [H]\leq[K]\}$ be the corresponding subposet of the conjugacy classes $[K]$ with $[H]\leq[K]$.
Suppose that, for every $K\in\cS\setminus\{G\}$, we have $N_{G^\prime}(K)=G^\prime \cap K$; this holds for instance if $G$ is perfect and every $K\in\cS$ is self-normalizing in $G$.
Then the $(\mu,\lambda)$-property for $H$ holds if and only $\mu(H)=\lambda(H)$, and hence if and only if the posets $\cS$ and $\bar{\cS}$ are isomorphic.

In the case $S_4\cong H\leq G= U_3(3)$, we have that $\cS=\{H,M_1,M_2,M_3,G\}\not\cong\bar{\cS}=\{[H],[M_1],[M_2]=[M_3],[G]\}$, with $M_i$ is a maximal subgroup of $G$; for every $K\in\cS$, $K$ is self-normalizing in the simple group $G$. Thus, $\cS\not\cong\bar{\cS}$ implies that the $(\mu,\lambda)$-property fails at $H$.
\end{remark}

%


\begin{center}
\begin{table}[!ht]
\caption{Subgroups $H\leq G=L_2(q)$, $q=2^e$, with $\mu(H)\ne0$ or $\lambda(H)\ne0$}\label{tabella:l2qpari}
\def\arraystretch{1.3}
\begin{tabular}{|c| c c c c c c c|}
\hline
$H$ & $\mathcal{S}_h$ & $M_{h,2r(h)}$ & $D_{4r(h)}$ & $D_{4s(h)}$ & $C_{2r(h)}$ & $C_2$ & $\{1\}$ \\ \hline
$\forall h\mid e$\textrm{ with} & $h>1$ & $h>1$ & $h>1$ & $h>1$ & $h>1$ & $-$ & $-$ \\
$\mu(H)$ & $\mu\left(\frac{e}{h}\right)$ & $-\mu\left(\frac{e}{h}\right)$ & $-\mu\left(\frac{e}{h}\right)$ & $-\mu\left(\frac{e}{h}\right)$ & $\frac{2(q-1)}{2^h-1}\mu\left(\frac{e}{h}\right)$ & $-q\mu(e)$ & $q(q^2-1)\mu(e)$ \\
$N_G(H)$ & $H$ & $H$ & $H$ & $H$ & $D_{2(q-1)}$ & $E_q$ & $G$  \\
$\lambda(H)$ & $\mu\left(\frac{e}{h}\right)$ & $-\mu\left(\frac{e}{h}\right)$ & $-\mu\left(\frac{e}{h}\right)$ & $-\mu\left(\frac{e}{h}\right)$ & $\mu\left(\frac{e}{h}\right)$ & $-2\mu(e)$ & $\mu(e)$ \\
\hline
\end{tabular}
\end{table}
\end{center}

\begin{center}
\begin{table}[!ht]
\caption{Subgroups $H\leq G=L_2(q)$ with $\mu(H)\ne0$ or $\lambda(H)\ne0$; $q=p^e$; $p$ odd; $e>2$, or $e=2$ and $p\equiv\pm1\pmod5$; $H\notin\{D_4,C_3,C_2,\{1\}\}$.}\label{tabella:l2qdispari}
\def\arraystretch{1.3}
\begin{tabular}{|c| c c c c c|}
\hline
$H$ & $\mathcal{G}_h$ & $\mathcal{S}_h$ & $M_{h,2r(h)}$ & $M_{h,r(h)}$ & $D_{4r(h)}$ \\ \hline
$\forall h\mid e$ with & $\frac{e}{h}$ even & $\frac{e}{h}$ odd & $\frac{e}{h}$ even & $\frac{e}{h}$ odd & $\frac{e}{h}$ even, $p^h\ne3$ \\
$\mu(H)$ & $\mu\left(\frac{e}{h}\right)$ & $\mu\left(\frac{e}{h}\right)$ & $-\mu\left(\frac{e}{h}\right)$ & $-\mu\left(\frac{e}{h}\right)$ & $-2\mu\left(\frac{e}{h}\right)$ \\
$N_G(H)$ & $H$ & $H$ & $H$ & $H$ & $D_{\gcd\left(8r(h),q-1\right)}$ \\
$\lambda(H)$ & $\mu\left(\frac{e}{h}\right)$ & $\mu\left(\frac{e}{h}\right)$ & $-\mu\left(\frac{e}{h}\right)$ & $-\mu\left(\frac{e}{h}\right)$ & $-\frac{2\mu(e/h)}{\gcd\left(2,\frac{q-1}{4r(h)}\right)}$ \\
\hline
$H$ & $D_{2r(h)}$ & $D_{4s(h)}$ & $D_{2s(h)}$ & $C_{2r(h)}$ & $C_{r(h)}$ \\ \hline
$\forall h\mid e$ with & $\frac{e}{h}$ odd, $p^h\notin\{3,5\}$ & $\frac{e}{h}$ even & $\frac{e}{h}$ odd, $p^h\ne3$ & $\frac{e}{h}$ even, $p^h\ne3$ & $\frac{e}{h}$ odd, $p^h\notin\{3,5\}$ \\
$\mu(H)$ & $-\mu\left(\frac{e}{h}\right)$ & $-2\mu\left(\frac{e}{h}\right)$ & $-\mu\left(\frac{e}{h}\right)$ & $\frac{2(q-1)}{p^h-1}\mu\left(\frac{e}{h}\right)$ & $\frac{2(q-1)}{p^h-1}\mu\left(\frac{e}{h}\right)$ \\
$N_G(H)$ & $H$ & $D_{\gcd\left(8s(h),q-1\right)}$ & $H$ & $D_{q-1}$ & $D_{q-1}$ \\
$\lambda(H)$ & $-\mu\left(\frac{e}{h}\right)$ & $-\frac{2\mu(e/h)}{\gcd\left(2,\frac{q-1}{4s(h)}\right)}$ & $-\mu\left(\frac{e}{h}\right)$ & $2\mu\left(\frac{e}{h}\right)$ & $\mu\left(\frac{e}{h}\right)$ \\
\hline
\end{tabular}
\end{table}
\end{center}

\begin{center}
\begin{table}[!ht]
\caption{Subgroups $H< G=L_2(q)$ with $\mu(H)\ne0$ or $\lambda(H)\ne0$; $q=p^e$; $p$ odd; $e>2$, or $e=2$ and $p\equiv\pm1\pmod5$; $H\in\{ D_4,C_3,C_2,\{1\}\}$.}\label{tabella:l2qdispariecc}
\def\arraystretch{1.3}
\begin{tabular}{|c| c c|}
\hline
$H$ & $D_4$ & $C_3$\\
\hline
$\mu(H)$ & $\begin{array}{l} \alpha-\beta,\;{\rm where } \; \alpha=\begin{cases} -6,&e=2^r; \\ 0,&\textrm{otherwise}; \end{cases} \\ \beta=\begin{cases} 6\mu(e),&p=3,\,e\textrm{ even}, \\ 3\mu(e),&p\in\{3,5\},e\textrm{ odd}; \\ 0,&\textrm{otherwise}. \end{cases} \end{array}$ & $\begin{cases} \frac{(q-1)}{3}\mu(e),& p=7,\,e\textrm{ odd}; \\ \frac{q}{3},& p=3; \\ 0,&\textrm{otherwise}. \end{cases}$ \\
$N_G(H)$ & $\begin{cases} S_4,& q\equiv\pm1\pmod8; \\ A_4,& \textrm{otherwise}. \end{cases}$ & $\begin{cases} E_q,&q\equiv0\pmod3; \\ D_{q-1},&q\equiv1\pmod3; \\ D_{q+1},&q\equiv2\pmod3. \end{cases}$\\
$\lambda(H)$ & $\begin{cases} \frac{1}{6}\mu(D_4),& q\equiv\pm1\pmod8; \\ \frac{1}{3}\mu(D_4),& \textrm{otherwise}. \end{cases}$ & $\begin{cases} \mu(e),& p=7,\,e\textrm{ odd}; \\ 1,& p=3; \\ 0,&\textrm{otherwise}. \end{cases}$ \\
\hline
$H$ & $C_2$ & $\{1\}$ \\
\hline
$\mu(H)$ & $\begin{array}{l} \gamma-\delta,\;{\rm where } \; \gamma=\begin{cases} \frac{1}{2}(q-1),&e=2^r; \\ 0,&\textrm{otherwise}; \end{cases} \\ \delta=\begin{cases} -(q-1)\mu(e),&p=3,\,e\textrm{ even}, \\ \frac{q+1}{2}\mu(e),&p=3,e\textrm{ odd}; \\ -\frac{q-1}{2}\mu(e),&p=5,e\textrm{ odd}; \\ 0,&\textrm{otherwise}. \end{cases} \end{array}$ & $\begin{cases} |G|\mu(e),&p=3,e\textrm{ odd}; \\ 0,&\textrm{otherwise}. \end{cases}$ \\
$N_G(H)$ & $\begin{cases} D_{q-1},&q\equiv1\pmod4; \\ D_{q+1},&q\equiv3\pmod4. \end{cases}$ & $G$ \\
$\lambda(H)$ & $\begin{cases} \frac{2\mu(C_2)}{q-1},&q\equiv1\pmod4; \\ \frac{2\mu(C_2)}{q+1},&q\equiv3\pmod4. \end{cases}$ & $\begin{cases} \mu(e),&p=3,e\textrm{ odd}; \\ 0,&\textrm{otherwise}. \end{cases}$ \\
\hline
\end{tabular}
\end{table}
\end{center}

\begin{center}
\begin{table}[!ht]
\caption{Subgroups $H< G=L_2(q)$ with $\mu(H)\ne0$ or $\lambda(H)\ne0$; $q=p^2$; $p\geq7$ odd; $p\equiv\pm2\pmod5$; $H$ non-maximal subgroup of $G$.}\label{tabella:l2qdispariquad}
\def\arraystretch{1.3}
\begin{tabular}{|c| c c c c c c|}
\hline
$H$ & $C_p : C_{p-1}$ & $C_{\frac{p^2-1}{2}}$ & $D_{2(p+1)}$ & $D_{2(p-1)}$ & $C_{p-1}$ & $A_4$ \\
\hline
$\mu(H)$ & $1$ & $2$ & $2$ & $2$ & $-2(p+1)$ & $2$ \\
$N_G(H)$ & $H$ & $D_{p^2-1}$ & $D_{(p+1)\gcd(4,p-1)}$ & $D_{(p-1)\gcd(4,p+1)}$ & $D_{p^2-1}$ & $S_4$ \\
$\lambda(H)$ & $1$ & $1$ & $\frac{2}{\gcd\left(2,\frac{p-1}{2}\right)}$ & $\frac{2}{\gcd\left(2,\frac{p+1}{2}\right)}$ & $-2$ & $1$ \\
\hline
$H$ & $D_{10}$ & $D_6$ & $D_4$ & $C_3$ & $C_2$ & $\{1\}$ \\
\hline
$\mu(H)$ & $2$ & $2$ & $-6$ & $-\frac{2}{3}(p^2-1)$ & $-\frac{3}{2}(p^2-1)$ & $0$ \\
$N_G(H)$ & $H$ & $D_{12}$ & $S_4$ & $D_{p^2-1}$ & $D_{p^2-1}$ & $G$ \\
$\lambda(H)$ & $2$ & $1$ & $-1$ & $-2$ & $-3$ & $0$ \\
\hline
\end{tabular}
\end{table}
\end{center}

\begin{center}
\begin{table}[!ht]
\caption{Subgroups $H\leq G=Sz(q)$, $q=2^e$, $e\geq3$ odd, with $\mu(H)\ne0$ or $\lambda(H)\ne0$.}\label{tabella:suzuki}
\def\arraystretch{1.3}
\begin{tabular}{|c| c c c c c c c c c|}
\hline
$H$ & $G(h)$ & $F(h)$ & $B_0(h)$ & $A_0(h)$ & $B_1(h)$ & $B_2(h)$ & $C_4$ & $C_2$ & $\{1\}$ \\
\hline
$\forall h\mid e$ with & $h>1$ & $h>1$ & $h>1$ & $h>1$ & $h>1$ & $h>1$ & $-$ & $-$ & $-$ \\
$\mu(H)$ & $\mu\left(\frac{e}{h}\right)$ & $-\mu\left(\frac{e}{h}\right)$ & $-\mu\left(\frac{e}{h}\right)$ & $2\frac{q-1}{2^h-1}\mu\left(\frac{e}{h}\right)$ & $-\mu\left(\frac{e}{h}\right)$ & $-\mu\left(\frac{e}{h}\right)$ & $-q\mu(e)$ & $-\frac{q^2}{2}\mu(e)$ & $|G|\mu(e)$ \\
$N_G(H)$ & $H$ & $H$ & $H$ & $D_{2(q-1)}$ & $H$ & $H$ & $E_q\,.\,C_2$ & $E_q^{1+1}$ & $G$ \\
$\lambda(H)$ & $\mu\left(\frac{e}{h}\right)$ & $-\mu\left(\frac{e}{h}\right)$ & $-\mu\left(\frac{e}{h}\right)$ & $\mu\left(\frac{e}{h}\right)$ & $-\mu\left(\frac{e}{h}\right)$ & $-\mu\left(\frac{e}{h}\right)$ & $-2\mu(e)$ & $-\mu(e)$ & $\mu(e)$ \\
\hline
\end{tabular}
\end{table}
\end{center}

\begin{center}
\begin{table}[!ht]
\caption{Subgroups $H\leq G=R(q)$, $q=3^e$, $e\geq3$ odd, with $\mu(H)\ne0$ or $\lambda(H)\ne0$.}\label{tabella:ree}
\def\arraystretch{1.3}
\begin{tabular}{|c| c c c c c|}
\hline
$H$ & $R(3^h)$ & $(3^h+3^{\frac{h+1}{2}}+1):6$ & $(3^h-3^{\frac{h+1}{2}}+1):6$ & $(3^h)^{1+1+1}:(3^h-1)$ & $2\times L_2(3^h)$ \\
\hline
$\forall h\mid e$ with & $-$ & $-$ & $h>1$ & $-$ & $h>1$ \\
$\mu(H)$ & $\mu\left(\frac{e}{h}\right)$ & $-\mu\left(\frac{e}{h}\right)$ & $-\mu\left(\frac{e}{h}\right)$ & $-\mu\left(\frac{e}{h}\right)$ & $-\mu\left(\frac{e}{h}\right)$ \\
$N_G(H)$ & $H$ & $H$ & $H$ & $H$ & $H$ \\
$\lambda(H)$ & $\mu\left(\frac{e}{h}\right)$ & $-\mu\left(\frac{e}{h}\right)$ & $-\mu\left(\frac{e}{h}\right)$ & $-\mu\left(\frac{e}{h}\right)$ & $-\mu\left(\frac{e}{h}\right)$ \\
\hline
$H$ &  $2\times\left(3^h:\frac{3^h-1}{2}\right)$ & $\left(2^2\times D_{\frac{3^h+1}{2}}\right):3$ & $2^2\times D_{\frac{3^h+1}{2}}$ & $2\times L_2(3)$ & $2^3$ \\
\hline
$\forall h\mid e$ with & $h>1$ & $h>1$ & $h>1$ & $-$ & $-$ \\
$\mu(H)$ & $\mu\left(\frac{e}{h}\right)$ & $-\mu\left(\frac{e}{h}\right)$ & $3\mu\left(\frac{e}{h}\right)$ & $-2\mu(e)$ & $21\mu(e)$ \\
$N_G(H)$ & $H$ & $H$ & $H$ & $H$ & ${\rm A\Gamma L}_1(8)$ \\
$\lambda(H)$ & $\mu\left(\frac{e}{h}\right)$ & $-\mu\left(\frac{e}{h}\right)$ & $3\mu\left(\frac{e}{h}\right)$ & $-2\mu\left(\frac{e}{h}\right)$ & $\mu\left(\frac{e}{h}\right)$ \\
\hline
\end{tabular}
\end{table}
\end{center}

\section{Acknowledgement}

This research was partially supported by the Italian National Group for Algebraic and Geometric Structures and their Applications (GSAGA - INdAM). 
The use of the software Magma \cite{MAGMA} for the computation of the functions $\mu$ and $\lambda$ of a given group, although not necessary for the theoretical proofs, was very useful to the authors during the work in order to test ideas or disprove conjectures.

\end{document}